\documentclass[a4paper, reqno, 12pt]{amsart}

\usepackage[usenames,dvipsnames]{color}
\usepackage{amsthm,amsfonts,amssymb,amsmath,amsxtra}
\usepackage[all]{xy}
\SelectTips{cm}{}
\usepackage{xr-hyper}
\usepackage[colorlinks=true,
   citecolor=Black,
   linkcolor=Red,
   urlcolor=Blue]{hyperref}
\usepackage{verbatim}

\usepackage[margin=1.25in]{geometry}
\usepackage{mathrsfs}

\RequirePackage{xspace}
\RequirePackage{etoolbox}
\RequirePackage{varwidth}
\RequirePackage{enumitem}
\RequirePackage{tensor}
\RequirePackage{mathtools}
\RequirePackage{longtable}
\RequirePackage{multirow}
\usepackage{systeme}

\def\ge{\geqslant}
\def\le{\leqslant}
\def\a{\alpha}
\def\b{\beta}
\def\g{\gamma}

\def\D{\Delta}

\def\e{\epsilon}

\def\o{\omega}
\def\p{\pi}

\def\s{\sigma}

\def\l{\lambda}

\def\i{^{-1}}

\def\<{\langle}
\def\>{\rangle}

\newcommand{\BB}{\ensuremath{\mathbb {B}}\xspace}
\newcommand{\BC}{\ensuremath{\mathbb {C}}\xspace}

\newcommand{{\BG}}{\ensuremath{\mathbb {G}}\xspace}

\newcommand{{\BK}}{\ensuremath{\mathbb {K}}\xspace}

\newcommand{\BN}{\ensuremath{\mathbb {N}}\xspace}

\newcommand{\BQ}{\ensuremath{\mathbb {Q}}\xspace}
\newcommand{\BR}{\ensuremath{\mathbb {R}}\xspace}

\newcommand{\BZ}{\ensuremath{\mathbb {Z}}\xspace}

\newcommand{\CB}{\ensuremath{\mathcal {B}}\xspace}

\newcommand{\CI}{\ensuremath{\mathcal {I}}\xspace}

\newcommand{\CT}{\ensuremath{\mathcal {T}}\xspace}

\newcommand{\Irr}{\mathrm{Irr}}

\DeclareMathOperator{\rank}{rank}

\newcommand{\reg}{{\mathrm{reg}}}

\DeclareMathOperator{\supp}{supp}

\newtheorem{theorem}{Theorem}
\newtheorem{proposition}[theorem]{Proposition}
\newtheorem{lemma}[theorem]{Lemma}
\newtheorem {conjecture}[theorem]{Conjecture}

\theoremstyle{definition}

\numberwithin{equation}{section}
\numberwithin{theorem}{section}

\setitemize[0]{leftmargin=*,itemsep=\the\smallskipamount}
\setenumerate[0]{leftmargin=*,itemsep=\the\smallskipamount}

\renewcommand{\to}{%
   \ifbool{@display}{\longrightarrow}{\rightarrow}%
   }
\let\shortmapsto\mapsto
\renewcommand{\mapsto}{%
   \ifbool{@display}{\longmapsto}{\shortmapsto}%
   }
\newlength{\olen}
\newlength{\ulen}
\newlength{\xlen}
\newcommand{\xra}[2][]{%
   \ifbool{@display}%
      {\settowidth{\olen}{$\overset{#2}{\longrightarrow}$}%
       \settowidth{\ulen}{$\underset{#1}{\longrightarrow}$}%
       \settowidth{\xlen}{$\xrightarrow[#1]{#2}$}%
       \ifdimgreater{\olen}{\xlen}%
          {\underset{#1}{\overset{#2}{\longrightarrow}}}%
          {\ifdimgreater{\ulen}{\xlen}%
             {\underset{#1}{\overset{#2}{\longrightarrow}}}
             {\xrightarrow[#1]{#2}}}}%
      {\xrightarrow[#1]{#2}}
   }
\makeatother
\newcommand{\xyra}[2][]{%
   \settowidth{\xlen}{$\xrightarrow[#1]{#2}$}%
   \ifbool{@display}%
      {\settowidth{\olen}{$\overset{#2}{\longrightarrow}$}%
       \settowidth{\ulen}{$\underset{#1}{\longrightarrow}$}%
       \ifdimgreater{\olen}{\xlen}%
          {\mathrel{\xymatrix@M=.12ex@C=3.2ex{\ar[r]^-{#2}_-{#1} &}}}%
          {\ifdimgreater{\ulen}{\xlen}%
             {\mathrel{\xymatrix@M=.12ex@C=3.2ex{\ar[r]^-{#2}_-{#1} &}}}
             {\mathrel{\xymatrix@M=.12ex@C=\the\xlen{\ar[r]^-{#2}_-{#1} &}}}}}%
      {\mathrel{\xymatrix@M=.12ex@C=\the\xlen{\ar[r]^-{#2}_-{#1} &}}}%
   }
\makeatletter
\newcommand{\xla}[2][]{%
   \ifbool{@display}%
      {\settowidth{\olen}{$\overset{#2}{\longleftarrow}$}%
       \settowidth{\ulen}{$\underset{#1}{\longleftarrow}$}%
       \settowidth{\xlen}{$\xleftarrow[#1]{#2}$}%
       \ifdimgreater{\olen}{\xlen}%
          {\underset{#1}{\overset{#2}{\longleftarrow}}}%
          {\ifdimgreater{\ulen}{\xlen}%
             {\underset{#1}{\overset{#2}{\longleftarrow}}}
             {\xleftarrow[#1]{#2}}}}%
      {\xleftarrow[#1]{#2}}
   }
\newcommand{\isoarrow}{%
   \ifbool{@display}{\overset{\sim}{\longrightarrow}}{\xrightarrow\sim}%
   }
   
\begin{document}

\title[]{Total positivity and conjugacy classes}

\author{Xuhua He}
\address{X.~H., The Institute of Mathematical Sciences and Department of Mathematics, The Chinese University of Hong Kong, Shatin, N.T., Hong Kong SAR, China}
\email{xuhuahe@math.cuhk.edu.hk}
\author{George Lusztig}
\address{G.~L., Department of Mathematics, M.I.T., Cambridge, MA 02139, USA}
\email{gyuri@mit.edu}

\thanks{}



\begin{abstract}
In this paper we study the interaction between the totally positive monoid $G_{\ge 0}$ attached to a connected reductive group $G$ with a pinning and the conjugacy classes in $G$. In particular we study how a conjugacy class meets the various cells of $G_{\ge0}$. We also state a conjectural Jordan decomposition for $G_{\ge0}$ and prove it in some special cases.
\end{abstract}

\maketitle

\section*{Introduction}

Let $G$ be a connected reductive algebraic group over $\BC$. Let $G_{>0}\subset G_{\ge0}$ be the totally positive submonoids of $G$ associated to a pinning of $G$ in \cite{Lu94}. (In the case where $G=GL_n(\BC)$, the definition of $G_{>0},G_{\ge0}$ goes back to Schoenberg and to Gantmacher-Krein in the 1930's.) In this paper we are mainly interested in the interaction between these monoids and the conjugacy classes of $G$. Here are some examples of such interaction in the earlier literature.

In \cite{GK}, Gantmacher and Krein showed that when $G=GL_n(\BC)$, any matrix in $G_{>0}$ has $n$ distinct eigenvalues which are all in $\BR_{>0}$. This result was generalized in \cite{Lu94} to a general $G$ by showing that 

(a) any $g\in G_{>0}$ is regular semisimple and contained in a $\BR$-split maximal torus.

In \cite{Cr76}, it is shown that when $G=GL_n(\BC)$, any matrix in $G_{\ge0}$ has all eigenvalues in $\BR_{>0}$. This result was generalized in \cite{Lu19} to a general $G$ by showing that 

(b) any $g\in G_{\ge0}$ acts on any finite dimensional representation of $G$ with all eigenvalues in $\BR_{>0}$. 

In \cite{Lu94} it is shown that $G_{\ge0}$ contains ``relatively few'' unipotent elements in the sense that

(c) if $C$ is a unipotent class in $G$ then $C\cap G_{\ge0}$ has dimension less than or equal to half the dimension of the set of real points of $C$ (and it can be empty).

Let $G^{\reg}$ be the open subset of $G$ consisting of regular elements in the sense of Steinberg. Let $G^{\reg, \text{ss}}$ be the open subset of $G$ consisting of regular semisimple elements. 

In \cite{Lu94} a partition of $G_{\ge0}$ into cells indexed by two Weyl group elements was defined. In this paper we show that each cell of $G_{\ge0}$ contains regular elements (and even regular semisimple elements). We describe explicitly the cells of $G_{\ge0}$ which are entirely contained in $G^{reg}$ and also the cells which are entirely contained in $G^{\reg, \text{ss}}$.

We show that if $G$ is almost simple, then the set of regular unipotent elements in $G_{\ge0}$ is contained in the union of the unipotent radicals of the two Borel subgroups which are part of the pinning.

If $g\in G$, we have a Jordan decomposition $g=g_sg_u=g_ug_s$ with $g_u$ unipotent and $g_s$ semisimple. Let $H$ be the centralizer of $g_s$ in $G$ (by \cite{Lu19}, $H$ is a connected reductive group). If $g\in G_{\ge0}$ then $g_u,g_s$ are not necessarily contained in $G_{\ge0}$. 

As a substitute, in this paper we state a conjecture which says that $H$ has something close to a pinning with respect to which we can define $H_{\ge0}$ and we have $g_u\in H_{\ge0},g_s\in H_{\ge0}$. This conjecture is proved in some cases in this paper. 

A consequence of this conjecture and of (c) is that, if $C'$ is a non-semisimple conjugacy class in $G$, defined over $\BR$, then $C'\cap G_{\ge0}$ has dimension less than the dimension of the set of real points of $C'$.

\section{Pinnings and total positivity}

\subsection{Totally nonnegative monoid}
We assume that $G$ is defined and split over $\BR$. We fix a pinning $\mathbf P=(B^+, B^-, T, x_i, y_i; i \in I)$. Let $W=N_G(T)/T$ be the Weyl group of $G$ and $\{s_i\}_{i \in I}$ be the set of simple reflection in $W$. Let $G_{\ge 0}$ be the totally nonnegative submonoid of $G$ introduced in \cite{Lu94}. Let $w \in W$ and $w=s_{i_1} s_{i_2} \cdots s_{i_n}$ be a reduced expression of $w$. We set 
\begin{align*} U^+_{w, >0} &=\{x_{i_1}(a_1) x_{i_2}(a_2) \cdots x_{i_n}(a_n); a_1, \ldots, a_n>0\}; \\
U^-_{w, >0} &=\{y_{i_1}(a_1) y_{i_2}(a_2) \cdots y_{i_n}(a_n); a_1, \ldots, a_n>0\}.
\end{align*} 
For any $w_1, w_2 \in W$, we set $G_{w_1, w_2, >0}=U^+_{w_1, >0} T_{>0} U^-_{w_2, >0}$. By \cite[\S 2.11]{Lu94}, $G_{w_1, w_2, >0}=U^-_{w_2, >0} T_{>0} U^+_{w_1, >0}$ and $$G_{\ge 0}=\sqcup_{w_1, w_2 \in W} G_{w_1, w_2, >0}.$$

For any $i \in I$, let $\a_i$ be the corresponding simple root. Set $\dot s_i=x_i(1) y_i(-1) x_i(1)$. For any $w \in W$, it is known that for any reduced expression $w=s_{i_1} s_{i_2} \cdots$ of $w$, the element $\dot s_{i_1} \dot s_{i_2} \cdots$ is independent of the choice of the reduced expression. We denote this element by $\dot w$. For any element $w \in W$, let $\supp(w)$ be the set of simple reflections occurring in some (or equivalently, any) reduced expression of $w$.

For any $J \subset I$, let $P^+_J \supset B^+$ be the standard parabolic subgroup of $G$, $P^-_J \supset B^-$ be the opposite parabolic subgroup and $L_J=P^+_J \cap P^-_J$ be the standard Levi subgroup. For any parabolic subgroup $P$, let $U_P$ be its unipotent radical. Then $P^\pm_J=L_J U_{P^\pm_J}$. Let $\Phi$ be the root system of $G$ and $\Phi_J \subset \Phi$ be the subsystem attached to $L_J$. The Weyl group $W_J$ of $L_J$ is the subgroup of $W$ generated by $s_j$ for $j \in J$. We denote by $w_J$ the longest element in $W_J$. We simply write $U^{\pm}_{>0}$ for $U^{\pm}_{w_I, >0}$ and write $G_{>0}$ for $G_{w_I, w_I, >0}$. The set $G_{>0}$ is the totally positive part of $G$. 

For any $J \subset I$, let $W^J$ (resp. ${}^J W$) be the set of minimal length elements in their cosets in $W/W_J$ (resp. $W_J \backslash W$). For any $J, J' \subset I$, we simply write ${}^J W^{J'}$ for ${}^J W \cap W^{J'}$. 

\subsection{Weak pinnings}\label{sec:weakp}
We define an equivalence relation on the set of pinnings of $G$. Two pinnings $\mathbf P=(B^+,B^-,T,x_i,y_i; i\in I)$, $\mathbf P'=(B'{}^+,B'{}^-,T',x'_i,y'_i; i\in I')$ are said to be equivalent if $T=T',I=I'$ and if there exists $c=(c_i)_{i\in I}\in\BR_{>0}^I$ such that for any almost simple ``factor'' $G_1$ of the derived subgroup of $G$ we have 
\begin{itemize}
    \item either $B^+\cap G_1=B'{}^+\cap G_1$, $B^-\cap G_1=B'{}^-\cap G_1$, $x'_i(a)=x_i(c_ia)$, $y'_i(a)=y_i(c_i\i a)$ for all $i\in I,a\in\BC$; 
    
    \item or $B^+\cap G_1=B'{}^-\cap G_1$, $B^-\cap\ G_1=B'{}^+\cap G_1$, $x'_i(a)=y_i(c_ia)$, $y'_i(a)=x_i(c_i\i a)$ for all $i\in I,a\in\BC$.
\end{itemize}

This is an equivalence relation on the set of pinnings of $G$. An equivalence class of pinnings is called a {\it weak pinning} of $G$. We denote by $\underline{\mathbf P}$ the weak pinning containing a pinning $\mathbf P$.

The monoids $G_{\ge0},G_{>0}$ attached in \cite{Lu94} to a pinning of $G$ depend only on the equivalence class of that pinning hence can be viewed as being associated to a weak pinning of $G$. To show that the monoid $G_{\ge 0}$ determines a unique weaking pinning of $G$, we first prove the following lemma. 

\begin{lemma}\label{sec:reg-uni}
Suppose that $G$ is almost simple. Let $G^{\reg, \text{uni}}$ be the set of regular unipotent elements in $G$. Then $$G^{\reg, \text{uni}} \cap G_{\ge 0}=\bigsqcup_{w \in W, \supp(w)=I} U^+_{w, >0} \sqcup \sqcup_{w \in W, \supp(w)=I} U^-_{w, >0}.$$

In particular, $G^{\reg, \text{uni}} \cap G_{\ge 0}$ has two connected components and they are dense in the unipotent radical of two opposite Borel subgroups $B^+$ and $B^-$. 
\end{lemma}

\begin{proof}
Let $G^{\text{uni}}$ be the set of unipotent elements in $G$. By \cite[Theorem 6.6]{Lu94}, $$G^{\text{uni}} \cap G_{\ge 0}=\bigsqcup_{w_1, w_2 \in W; \supp(w_1) \cap \supp(w_2)=\emptyset} U^+_{w_1, >0} U^-_{w_2, >0}.$$ 

Suppose that $G$ is almost simple. Let $w_1, w_2 \in W$ and $g \in U^+_{w_1, >0} U^-_{w_2, >0}$. Let $J_i=\supp(w_i)$ for $i=1, 2$. We assume furthermore that $J_1 \cap J_2=\emptyset$. Then $\dot w_{J_2} g \dot w_{J_2} \i \in U^+$. 
If $J_2\neq \emptyset$ and $J_2 \neq I$, then there exists $i \in I-J_2$ such that in the Dynkin diagram of $G$, the vertex corresponds $i$ is connected to some vertex in $J_2$. In this case, $w_{J_2}(\Phi_{J_1}) \cap \Phi_{\{i\}}=\emptyset$ and $\dot w_{J_2} g \dot w_{J_2} \i \in U_{P^+_{I-\{i\}}}$. Hence $g$ is not regular. 

Similarly, if $J_1 \neq \emptyset$ and $J_1 \neq I$, then $g$ is not regular. It is obvious that if $J_1=J_2=\emptyset$, then $g=1$ and hence is not regular. Finally if $(J_1, J_2)=(I, \emptyset)$ or $(\emptyset, I)$, then $g$ is regular. 

The ``in particular'' part follows from that the fact that $\sqcup_{w \in W, \supp(w)=I} U^\pm_{w, >0}$ is connected and dense in $U^\pm$. 
\end{proof}

\subsection{Uniqueness of weak pinnings}
Let $G_{\ge 0}$ be the totally nonnegative part of $G$ associated to pinnings $\mathbf P$ and $\mathbf P'=(B'{}^+, B'{}^-, T', x'_i, y'_i; i \in I')$. We show that $\mathbf P$ and $\mathbf P'$ are equivalent. 

Let $G_{\text{der}}$ be the derived subgroup of $G$. Let $G_k$ for $1 \le k \le l$ be the almost simple factors of $G_{\text{der}}$. 
 
For any $k$, $G_k \cap G_{\ge 0}$ is a totally nonnegative part of $G_k$ associated to the pinnings $(B^+ \cap G_k, B^- \cap G_k, T \cap G_k, \ldots)$, $(B'{}^+ \cap G_k, B'{}^- \cap G_k, T' \cap G_k, \ldots)$ of $G_k$. By Lemma \ref{sec:reg-uni},  $G_k \cap G'_{\ge 0}$ has two connected components and the closure of these connected components are the unipotent radicals of two opposite Borel subgroups $\{B^+ \cap G_k, B^- \cap G_k\}=\{B'{}^+ \cap G_k, B'{}^- \cap G_k\}$. In particular, $$T \cap G_k=(B^+ \cap G_k) \cap (B^- \cap G_k)=(B'{}^+ \cap G_k) \cap (B'{}^+ \cap G_k)=T' \cap G_k.$$ Hence $T=(T \cap G_1) \cdots (T \cap G_l) Z(G)=(T' \cap G_1) \cdots (T' \cap G_l) Z(G)=T'$. 
Recall that $\Phi$ is the set of roots of $G$ with respect to $T$. Then $\Phi=\Phi_1 \sqcup \cdots \sqcup \Phi_l$, where $\Phi_k$ is the set of roots of $G_k$ with respect to $T \cap G_k$. Let $\Sigma$ be the set of positive roots determined by $B^+$ and $\Sigma'$ be the set of positive roots determined by $B'{}^+$. Since $\{B^+ \cap G_k, B^- \cap G_k\}=\{B'{}^+ \cap G_k, B'{}^- \cap G_k\}$, we have $\Sigma' \cap \Phi_k=\Sigma \cap \Phi_k$ or $-\Sigma \cap \Phi_k$. In particular, we may identify $I$ with $I'$. If $B'{}^+ \cap G_k=B^+ \cap G_k$, then for any $i \in I$ with $\a_i \in \Phi_K$, we have $x'_i(\BR_{>0})=x_i(\BR_{>0})$ and thus there exists $c_i>0$ such that $x'_i(a)=x_i(c_i a)$ for all $a \in \BC$. Similarly, if  $B'{}^+ \cap G_k=B^- \cap G_k$, then for any $i \in I$ with $\a_i \in \Phi_K$, we have $x'_i(\BR_{>0})=y_i(\BR_{>0})$ and thus there exists $c_i>0$ such that $x'_i(a)=y_i(c_i a)$ for all $a \in \BC$. 

Thus $\mathbf P'$ and $\mathbf P$ are equivalent. 

\subsection{Weak bases}\label{sec:weakb}
Let $V$ be a $\BC$-vector space. A {\it half line} in $V$ is a subset of $V$ of the form $\{av;a\in\BR_{\ge0}\}$ for some $v\in V-\{0\}$. Assume now that $\dim(V)<\infty$. A {\it weak basis} of $V$ is a collection of half lines in $V$ such that if we choose a nonzero vector on each of these half lines, the resulting nonzero vectors form a basis of $V$. Note that any basis of $V$ gives rise to a weak basis of $V$, which consists of the half lines containing the various elements in that basis.

Let $\underline\b$ be a weak basis of $V$. Let $V_{\underline\b,\ge0}$ be the set of $\BR_{\ge0}$-linear combination of vectors in the half lines in $\b$. This subset is closed under addition and under scalar multiplication by $\BR_{\ge0}$. Let $V_{\underline\b,>0}$ be the set of $\BR_{>0}$-linear combination of nonzero vectors in the half lines in $\b$. This subset of $V$ is closed under addition and under scalar multiplication by $\BR_{>0}$. If $\underline\b,\underline\b'$ are two weak bases of $V$, we have $$V_{\underline\b,\ge0}=V_{\underline\b',\ge0}\Rightarrow V_{\underline\b,>0}=V_{\underline\b',>0}\Rightarrow \underline\b=\underline\b'.$$

We now assume that $G$ is simply laced.

We fix a weak pinning $\underline{\mathbf P}$ of $G$. Let $X_{\underline{\mathbf P}}$ be the set of Borel subgroups $B$ of $G$ such that $B=B^+$ for some pinning $\mathbf P=(B^+,B^-,\ldots)$ in $\underline{\mathbf P}$. 

Let $\Irr(G)$ be the collection of irreducible finite dimensional rational representations of $G$. Let $V\in\Irr G$. For any Borel subgroup $B$ of $G$ we denote by $V_B$ the unique $B$-invariant $\BC$-line in $V$. 

For $B\in X_{\underline{\mathbf P}}$ let $e_B$ be the set of half lines in $V_B$. For any pinning $\mathbf P=(B^+,B^-,\ldots)$ in $\underline{\mathbf P}$ and any $v\in V_{B^+}-\{0\}$ we denote by $\BB_{V,\mathbf P,v}$ the canonical basis of $V$ containing $v$ defined in terms of $\mathbf P$ in \cite{Lu90}.

If $\mathbf P'=(B'{}^+,B'{}^-,\ldots)$ is another pinning in $\underline{\mathbf P}$ and if $v' \in V_{B'{}^+}-\{0\}$ is defined by the condition that $v'\in\BB_{V,\mathbf P,v}$ then, using \cite[Proposition 21.1.2]{Lu93}, we see that the canonical basis $\BB_{V,\mathbf P',v'}$ of $V$ gives rise to the same weak basis of $V$ as $\BB_{V,\mathbf P,v}$; moreover, the half line in $V_{B'{}^+}$ containing $v'$ depends only on the half line containing $v$. Thus the sets $e_B$ for various $B\in X_{\underline{\mathbf P}}$ can be identified with a single set $e_{V,\underline{\mathbf P}}$. We also see that $\underline{\mathbf P}$ together with a choice of an element $\e\in e_{V,\underline{\mathbf P}}$ give rise to a weak basis $\b_{V,\underline{\mathbf P},\e}$ of $V$.

Hence the subsets $V_{\b_{V,\underline{\mathbf P}, \e},\ge0}$, $V_{\b_{V,\underline{\mathbf P}, \e},>0}$ are defined. 

According to \cite[\S 3.2]{Lu94}, if $g\in G_{\underline{\mathbf P}, >0}$ then $g:V\to V$ maps $V_{\b_{V,\underline{\mathbf P}, \e},\ge0}$ into $V_{\b_{V,\underline{\mathbf P}, \e},>0}$. Conversely, according to \cite[Theorem 1.11]{FZ}, if $g\in G$ is such that for any $V\in\text{Irr}(G)$, $g:V \to V$ maps $V_{\b_{V,\underline{\mathbf P},\e},\ge0}$ into $V_{\b_{V,\underline{\mathbf P}, \e},>0}$ (for some $\e$), then $g\in G_{>0}$.

\section{Regular elements}

\subsection{The main results}
Recall that $G^{\reg}$ is the set of regular elements in $G$ and $G^{\reg, \text{ss}}$ is the set of regular semisimple elements in $G$. Now we state the main results of this section. 

\begin{theorem}\label{main1}
Let $w_1, w_2 \in W$. Then 

(1) $G_{w_1, w_2, >0} \cap G^{\reg, \text{ss}} \neq \emptyset$. 

(2) $G_{w_1, w_2, >0} \subset G^{\reg, \text{ss}}$ if and only if $\supp(w_1)=\supp(w_2)=I$.  

(3) Suppose that $G$ is almost simple. Then $G_{w_1, w_2, >0} \subset G^{\reg}$ if and only if $\supp(w_1)=I$ or $\supp(w_2)=I$.  
\end{theorem}

\subsection{Some semigroups} Let $\ast: W \times W \to W$ is the monoid product (see \cite[\S 2.11]{Lu19}. Then for any $w, w' \in W$, we have $\supp(w \ast w')=\supp(w) \cup \supp(w')$. By \cite[\S 2.11(d)]{Lu19}, we have $G_{w_1, w_2, >0} G_{w_3, w_4, >0}=G_{w_1 \ast w_3, w_2 \ast w_4, >0}$. Now as a consequence of Theorem \ref{main1}, we have 

(a) {\it the union of the totally nonnegative cells consisting of regular semisimple elements is an open sub semigroup of $G_{\ge 0}$;} 

(b) {\it the union of the totally nonnegative cells consisting of regular elements is an open sub semigroup of $G_{\ge 0}$.}

\subsection{Oscillatory elements} We call an element $g \in G_{\ge 0}$ {\it oscillatory} if $g^m \in G_{>0}$ for some $m \in \BN$. By \cite[Proposition 2.19]{Lu94}, the subset of oscillatory elements is $G^{I, I}_{\ge 0}$. By \cite[Corollary 5.7]{Lu94}, any oscillatory element is regular and semisimple. This is the ``if'' part of Theorem \ref{main1} (2).

We have the following result, which generalizes \cite[Corollary 8.10]{Lu94} from totally positive elements to oscillatory elements. 

\begin{lemma}\label{lem:w-to-1}
Let $g$ be an oscillatory element. Then there exists a unique element $u \in U^-_{>0}$ such that $u \i g u \in B^+$. Moreover, $g=u u' t u \i$ for some $u' \in U^+_{\ge 0}$ and $t \in T$ with $\a_i(t)>1$ for all $i \in I$. 
\end{lemma}

\begin{proof}
Let $\CB=G/B^+$ be the flag variety of $G$. We denote the action of $G$ on $\CB$ by $g \cdot B:=g B g \i$. Let $\CB_{>0}=U^-_{>0} \cdot B^+ \subset \CB$ be the totally positive flag manifold and $\CB_{\ge 0}=\overline{\CB_{>0}}$ be the totally nonnegative flag manifold. By \cite[Corollary 8.11]{Lu94}, there exists $B \in \CB_{\ge 0}$ with $g \in B$. We have $g^m \in G_{>0}$ for $m \in \BN$. By \cite[Lemma 8.2]{Lu94}, there exists a unique $B' \in \CB_{\ge 0}$ with $g^n \in B'$. Since $g \in B$, we have $g^n \in B$. Thus $B=B'$. By \cite[Lemma 8.5]{Lu94}, $B' \in \CB_{>0}$. Hence $B=u B^+ u \i$ for some $u \in U^-_{>0}$. So $u \i g u \in B^+$. The uniqueness of $u$ follows from the uniqueness of $B$. 

Let $u \i g u=u' t$ with $u' \in U^+$ and $t \in T$. We then have $u u' t=g u \in G_{\ge 0}$. Therefore we have $u' \in U^+_{\ge 0}$ and $t \in T_{>0}$. Moreover, $g^m=u (u' t)^m u \i=u u'_m t^m u \i$ for some $u'_m \in U^+_{>0}$. By \cite[Corollary 8.10]{Lu94}, $\a_i(t^m)>1$ for all $i \in I$. Hence $\a_i(t)>1$ for all $i \in I$. 
\end{proof}

\subsection{Projection maps} 
For $J \subset I$, let $\pi^\pm_J: P^\pm_J \to L_J$ be the projection map. Then the restriction of $\pi^\pm_J$ to $T_{>0}$ is the identity map. Moreover, $\pi^+_J(y_i(a))=y_i(a)$ and $\pi^-_J(x_i(a))=x_i(a)$ for $i \in J$. We also have
\[
\pi^+_J(x_i(a))=\begin{cases} x_i(a), & \text{ if } i \in J \\ 1, & \text{ if } i \notin J \end{cases} \text{ and } \pi^-_J(y_i(a))=\begin{cases} y_i(a), & \text{ if } i \in J; \\ 1, & \text{ if } i \notin J. \end{cases}
\]

In particular, $\pi^\pm_J(P^\pm_J \cap G_{\ge 0})=L_{J, \ge 0}$. 
We define a morphism of monoids $\pi_J: W \to W_J$ by $s_i \mapsto s_i$ if $i \in J$ and $s_i \mapsto 1$ if $i \notin J$. Although not needed in this paper, it is worth pointing out that the map $W \to W_J$ can be regarded as the map $\pi^\pm_J: U^\pm(K) \to (U^\pm \cap L_J)(K)$ in the case where $K$ is the semifield of one element.

\subsection{Proof of Theorem \ref{main1} (1) and (2)}\label{only-if-1}
Let $J_i=\supp(w_i)$ for $i=1, 2$. Let $u_1 \in U^+_{w_1, >0}, u_2 \in U^-_{w_2, >0}$.

Set $K=J_1 \cap J_2$. Let $v_i=\pi_K(w_i)$. By our assumption, $\supp(v_1)=\supp(v_2)=K$. Let $u'_1=\pi^+_K(u_1)$ and $u'_2=\pi^-_K(u_2)$. Then $u'_1 \in U^+_{v_1, >0}$ and $u'_2 \in U^-_{v_2, >0}$. By Lemma \ref{lem:w-to-1}, there exists $u'' \in U^-_{w_K, >0}$ and $t_1 \in T_{>0}$ such that $(u'') \i u'_1 u'_2 u'' \in (U^+ \cap L_K) t_1$ and $\a_i(t_1)>1$ for all $i \in K$. Since $u_1 u_2 \in (U_{P^+_K} \cap L_{J_1}) u'_1 u'_2 (U_{P^-_K} \cap L_{J_2})$, we have 
\begin{align*} 
	(u'') \i u_1 u_2 u'' & \in (U_{P^+_K} \cap L_{J_1}) (u'') \i u'_1 u'_2 u'' (U_{P^-_K} \cap L_{J_2}) \\ & \subset (U_{P^+_K} \cap L_{J_1}) (U^+ \cap L_K) t (U_{P^-_K} \cap L_{J_2}) \\ 
	            &=(U^+ \cap L_{J_1}) t_1 (U_{P^-_K} \cap L_{J_2}).
\end{align*}

Set $g=\dot w_{J_2} \dot w_K$. We have $g (U_{P^-_K} \cap L_{J_2}) g \i \in U^+ \cap L_{J_2}$. Note that $w_{J_2} w_K \in W^{J_1}$. Then $g (U^+ \cap L_{J_1}) g \i \in U^+$. So $g (u'') \i u_1 u_2 u'' g \i \in U^+ (g t_1 g \i)$. 

Let $T'=\{t \in T_{>0}; \a_i(t)=1 \text{ for all } i \in K\}$. Then $g (u'') \i u_1 u_2 t u'' g \i \in U^+ (g t_1 t g \i)$ for any $t \in T'$. 

Note that there exists $t' \in T'$ such that $\a_i(t_1 t)>1$ for all $i \notin K$. Moreover, $\a_i(t_1 t)=\a_i(t_1)>1$ for $i \in K$. Thus $t_1 t$ is a regular semisimple element in $T$. So any element in $U^+ (g t_1 t g \i)$ is conjugate by an element in $U^+$ to $g t_1 t g\i$, which is regular semisimple. In particular, $u_1 u_2 t$ is regular semisimple. This proves part (1) of Theorem \ref{main1}. 

Now we come to the part (2) of Theorem \ref{main1}. The ``if'' part is proved in \cite{Lu94}. Now we prove the ``only if'' part. 

Without loss of generality, we assume that $J_2 \neq I$. Let $j \notin J_2$. There exists $t \in T'$ such that $\a_j(g t_1 t g \i)=1$. The element $u_1 u_2 t \in G_{w_1, w_2, >0}$ is conjugate to an element in $U^+ (g t_1 t g \i)$. Note that the semisimple part of any element in $U^+ (g t_1 t g \i)$ is conjugate by $U^+$ to $g t_1 t g \i$. In particular, if $u_1 u_2 t$ is semisimple, then it is conjugate to $g t_1 t g \i$. Since $\a_j(g t_1 t g \i)=1$, $g t_1 t g \i$ is not regular. Thus $u_1 u_2 t \notin G^{\reg, \text{ss}}$. 
\qed

\smallskip

\subsection{A technical lemma on the root system}
Before proving the ``only if'' part of Theorem \ref{main1}(3), we establish a lemma on the root system. 

\begin{lemma}\label{lem:root}
Suppose that $G$ is almost simple and $J, J'$ are proper subsets of $I$. Then there exists $j \notin J'$ such that $w_{J'}(\a_j) \notin \Phi_J$. 
\end{lemma}

\begin{proof}
If $J' \subset J$, then we may take $j$ to be any vertex not in $J$. 

Now assume that $J' \not \subset J$. Set $K=J \cap J'$. Then $J'-K \neq \emptyset$. We consider $J$ as a sub-diagram of the Coxeter diagram of $I$. Let $C$ be a connected component of $J'$ that contains some vertex $j_1$ in $J'-K$. Since $J' \subsetneq I$, there exists $j \notin J'$ such that $j$ is connected to a vertex $j'$ in $C$. 

Let $\o_{j_1}$ be the fundamental weight of $j_1$ with respect to $\Phi$ and $\o_{J', j_1}$ be the fundamental weight of $j_1$ with respect to $\Phi_{J'}$. Then $\o_{j_1} \in \sum_{i \in I} \BQ_{>0} \a_i$ and $\o_{J', j_1} \in \sum_{i \in C} \BQ_{>0} \a_i$. Note that $w_{J'}(\o_{J', j_1})=-\o_{J', j_2}=\o_{J', j_1}-(\o_{J', j_1}+\o_{J', j_2})$ for some $j_2 \in C$. We have $\o_{J', j_2} \in \sum_{i \in C} \BQ_{>0} \a_i$ and hence $\o_{J', j_1}+\o_{J', j_2} \in \sum_{i \in C} \BQ_{>0} \a_i$. Since the restriction of $\o_{j_1}$ to the root system $\Phi_{J'}$ equals to $\o_{J', j_1}$, we have $$w_{J'}(\o_{j_1})=\o_{j_1}-(\o_{J', j_1}+\o_{J', j_2}) \in \o_{j_1}-\sum_{i \in C} \BQ_{>0} \a_i.$$

Now $\<w_{J'}(\a_j), \o_{j_1}\>=\<\a_j, w_{J'}(\o_{j_1})\> \in \<\a_j,  \o_{j_1}-\sum_{i \in C} \BQ_{>0} \a_i\>=-\<\a_j, \sum_{i \in C} \BQ_{>0} \a_i\>$. Since $\<\a_j, \a_i\> \le 0$ for any $i \in C$ and $\<\a_j, \a_{j'}\><0$, we have $\<w_{J'}(\a_j), \o_{j_1}\>>0$. Since $j_1 \notin J$, we have $w_{J'}(\a_j) \notin \Phi_J$. 
\end{proof}

\subsection{Proof of the ``only if'' part of Theorem \ref{main1} (3)}
Let $J_i=\supp(w_i)$ for $i=1, 2$. Set $K=J_1 \cap J_2$. Let $T'=\{t \in T_{>0}; \a_i(t)=1 \text{ for all } i \in K\}.$ Let $u_1 \in U^+_{w_1, >0}, u_2 \in U^-_{w_2, >0}$. We show that there exists $t \in T'$ such that $u_1 u_2 t$ is not regular. Our strategy is similar to \S\ref{only-if-1} but the choice of $t$ is more involved. 

By \S\ref{only-if-1}, there exists $u'' \in U^-_{w_K, >0}$ and $t_1 \in T_{>0}$ such that $$(u'') \i u_1 u_2 u'' \in (U^+ \cap L_{J_1}) t_1 (U_{P^-_K} \cap L_{J_2}).$$ By Lemma \ref{lem:root}, there exists $j \notin J_2$ such that $w_K w_{J_2}(\a_j) \notin \Phi_{J_1}$. Set $g=\dot w_{J_2} \dot w_K$. We have $g (U_{P^-_K} \cap L_{J_2}) g \i \in U^+ \cap L_{J_2} \subset U_{P^+_{I-\{j\}}}$. Note that $w_{J_2} w_K \in W^{J_1}$, we have $g (U^+ \cap L_{J_1}) g \i \in U_{P^+_{I-\{j\}}}$. Hence $$g (U^+ \cap L_{J_1}) t_1 (U_{P^-_K} \cap L_{J_2}) g \i \subset U_{P^+_{I-\{j\}}} g t_1 g \i.$$ 

Let $t \in T'$ with $\a_j(g t_1 t g \i)=1$. Now 
\begin{align*}
g (u'') \i u_1 u_2 t u'' g \i & \in g (U^+ \cap L_{J_1}) t_1 t (U_{P^-_K} \cap L_{J_2}) g \i \subset U_{P^+_{I-\{j\}}} g t_1 t g \i. 
\end{align*}

The conjugation by $T$ preserves $U_{P^+_{I-\{j\}}} (g t_1 t g \i)$. Since $\a_j(g t_1 t g \i)=1$, the conjugation by the root subgroups $U_{\pm \a_j}$ preserves $U_{P^+_{I-\{j\}}} (g t_1 t g \i)$. Therefore, the $P^+_{I-\{j\}}$-conjugacy class $O_1$ of $g (u'') \i u_1 u_2 t u'' g \i$ is contained in $U_{P^+_{I-\{j\}}} (g t_1 t g \i)$. In particular, $\dim O_1 \le \dim U_{P^+_{I-\{j\}}}=\dim P^+_{I-\{j\}}-\rank G-2.$

Let $O$ be the conjugacy class of $u u' t$ in $G$. It is also the conjugacy class of $g (u'') \i u_1 u_2 t u'' g \i$ in $G$. We have $$\dim O \le \dim O_1+\dim G-\dim P^+_{I-\{j\}} \le \dim G-\rank G-2.$$ Therefore $u_1 u_2 t$ is not regular. 
\qed

\subsection{Regularity for $GL_n$}
Now we study the ``if'' part of Theorem \ref{main1} (3). The key part is to show that any element in $B^+_{>0}=T_{>0} U^+_{>0}$ is regular. In this subsection, we prove this claim via linear algebra. 

Let $G=GL_n$ and $g=(a_{ij})_{i,j \in[1,n]}$ in $B^+_{>0}$. In particular, $a_{ij}=0$ if $i<j$, $a_{ij}>0$ if $i\ge j$ and various minors are $>0$. For any $J_1, J_2 \subset \{1, 2, \ldots, n\}$, let $\D_{J_1, J_2}$ be the $(J_1, J_2)$-minor of $g$, i.e., the determinant of the submatrix $(a_{i j})_{i \in J_1, j \in J_2}$. To show that $g$ is regular, it suffices to show that any eigenspace of $g: \BR^n \to \BR^n$ is one-dimensional. 

\begin{lemma}
Let $1 \le k \le n$. Set $$V_k=\{(A_1,...,A_n)\in\BR^n;\sum_ja_{ij}A_j=a_{kk} A_i \text{ for all }i\in[1,n]\}.$$ Then $\dim V_k=1$.
\end{lemma}

\begin{proof} We argue by induction on $n$. 

If $n=1$, the result is obvious. 

Assume now that $n\ge2$. Let $c=a_{kk}$ and $I_c=\{i\in[1,n];a_{ii}=c\}$. Let $(A_1,...,A_n)\in V_k$. If $A_n=0$, then $(A_1,...,A_{n-1})$ satisfy a condition similar to that defining $V_k$. In this case, the statement follow from the induction hypothesis.

If $n\notin I_c$, then $(a_{nn}-c)A_n=0$ and $A_n=0$. 

We assume that $n\in I_c$. If $n-1\in I_c$, then $a_{n-1,n}A_n=0$. Since $a_{n-1,n}\ne0$, we have $A_n=0$ and the statement follows again from the induction hypothesis.

Now we assume that $n\in I_c,n-1\notin I_c$. If $n-2\in I_c$, then we have the linear system
\[\left\{
\begin{aligned} a_{n-2,n-1}A_{n-1}+a_{n-2,n}A_n &=0\\ (a_{n-1,n-1}-c)A_{n-1}+a_{n-1,n}A_n &=0
\end{aligned}
\right.
\]

The determinant of this linear system is
$$a_{n-2,n-1}a_{n-1,n}-(a_{n-1,n-1}-c)a_{n-2,n}=\D_{\{n-2, n-1\}, \{n-1, n\}}+ca_{n-2.n}>0.$$

It follows that $A_{n-1}=A_n=0$. In particular $A_n=0$ and we are done.

We then assume that $n\in I_c$, $n-1\notin I_c$, $n-2\notin I_c$. If $n-3\in I_c$, we have the linear system
\[\left\{
\begin{aligned} a_{n-3,n-2}A_{n-2}+a_{n-3,n-1}A_{n-1}+a_{n-3,n}A_n &=0\\ (a_{n-2,n-2}-c)A_{n-2}+a_{n-2,n-1}A_{n-1}+a_{n-2,n}A_n &=0 \\
(a_{n-1,n-1}-c)A_{n-1}+a_{n-1,n}A_n &=0
\end{aligned}
\right.
\]

The determinant of this system of linear equations is
\begin{align*} &a_{n-3,n-2}a_{n-2,n-1}a_{n-1,n}+(a_{n-2,n-2}-c)(a_{n-1,n-1}-c)a_{n-3,n}\\&
-a_{n-3,n-1}(a_{n-2,n-2}-c)a_{n-1,n}-a_{n-3,n-2}(a_{n-1,n-1}-c)a_{n-2,n}\\&
=
\D_{\{n-3, n-2, n-1\}, \{n-2, n-1, n\}}+c(\D_{\{n-3, n-1\}, \{n-1, n\}}+\D_{\{n-3, n-2\}, \{n-2, n\}})+c^2 a_{n-3,n}.
\end{align*}
Again the determinant is $>0$. It follows that $A_{n-2}=A_{n-1}=A_n=0$. In particular $A_n=0$ and we are done.

Thus we can assume $n\in I_c, n-1\notin I_c, n-2\notin I_c, n-3\notin I_c$.
Continuing in this way we see that we can assume $n\in I_c$ and  $n-1,n-2,...,1$ are not in $I_c$.

Then we can choose $A_n$ at random. From
$(a_{n-1,n-1}-c)A_{n-1}+a_{n-1,n}A_n=0$ where $a_{n-1,n-1}-c\neq 0$ we determine uniquely $A_{n-1}$.
From $(a_{n-2,n-2}-c)A_{n-2}+a_{n-2,n-1}A_{n-1}+a_{n-2,n}A_n=0$ where $a_{n-2,n-2}-c\neq 0$ we
determine uniquely $A_{n-2}$.
Continuing in this way we see that $A_{n-1},A_{n-2},...,A_1$ are uniquely determined
by $A_n$. Hence $\dim V_k=1$.
\end{proof}

\subsection{Conjugating the torus part} To handle the general case, we use a more Lie-theoretic approach.

For any $t \in T_{>0}$, let $I(t)=\{i \in I; \a_i(t)=1\}$. For any $t \in T_{>0}$ and $J \subset I$, let $\bar t_J$ be the unique element $t'$ in the $W_J$-orbit of $t$ with $\a_i(t') \ge 1$ for all $i \in J$. Let $w_{t, J}$ be the unique element in $W_{I(\bar t_J)} \backslash W_J/W_{I(t)}$ such that $\bar t_J=\dot w_{t, J} t \dot w_{t, J} \i$. 

\begin{lemma}\label{lem:conj}
Let $w \in W$ and $J=\supp(w)$. Then for any $t \in T_{>0}$ and $u \in U^+_{w, >0}$, there exists $u' \in U^-_{w_{t, J} \i, >0}$ such that $(u') \i t u u' \in \bar t_{J} U^+_{w, >0}$. 
\end{lemma}

\begin{proof}
We prove that 

(a) {\it Let $w \in W$, $u \in U^+_{w, >0}$ and $t \in T_{>0}$. Let $i \in J$. If $\a_i(t)<1$, then there exists $b>0$ such that $y_i(-b) t u y_i(b) \in s_i(t) U^+_{w, >0}$.}

Let $c=\a_i(t)<1$. We have $\pi^+_{\{i\}}(u)=x_i(a)$ for some $a>0$. Set $b=\frac{1-c}{c a}>0$. By \cite[\S 2.11]{Lu94}, $y_i(-b) t u y_i(b) \in y_i(-b) y_i(b') t' u'$ for some $b'>0$, $t' \in T_{>0}$ and $u' \in U^+_{w, >0}$. In particular, $y_i(-b) t u y_i(b) \in P^+_{\{i\}}$. Then \begin{align*} \pi^+_{\{i\}}(y_i(-b) t u y_i(b)) &=y_i(-b) t x_i(a) y_i(b)=t y_i(-\frac{1-c}{a}) x_i(a) y_i(\frac{1-c}{ca}) \\ &=t x_i(\frac{a}{c}) \a_i^\vee(\frac{1}{c}).\end{align*} Note that $t \a_i^\vee(\frac{1}{c})=s_i(t)$. (a) is proved. 

Now the statement follows from (a) by induction on $\ell(w_{t, \supp(w)})$. 
\end{proof}

\subsection{Proof of the ``if'' part of Theorem \ref{main1} (3)}
Let $g \in G_{w_1, w_2, >0}$. We assume that $\supp(w_1)=I$. The case where $\supp(w_2)=I$ is proved in a similar way. 

Let $J=\supp(w_2)$. By definition, $\pi^+_J(g) \in G_{\pi_J(w_1), w_2, >0} \subset L_J$. 

Since $\supp(\pi_J(w_1))=\supp(w_2)=J$, we have $\pi^+_J(g)$ is an oscillatory element of $L_{J, \ge 0}$. Thus by Lemma \ref{lem:w-to-1}, there exists $u \in U^-_{w_J, >0}$ such that $u \i \pi^+_J(g) u \in B^+ \cap L_J$. Then $u \i g u \in B^+$. Similar to the proof of Lemma \ref{lem:w-to-1}, we also have that $u \i g u \in G_{w_1, 1, >0}$. 
 
By Lemma \ref{lem:conj}, $u \i g u$ is conjugate to an element in $t \, U^+_{w_1, >0}$ for some $t \in T$ with $\a_i(t) \ge 1$ for all $i$. In particular, $Z_G(t)$ is a standard Levi subgroup $L_{J'}$ of $G$ for some $J' \subset I$. Note that $t \, U^+_{w_1, >0}=t \, U^+_{\pi_{J'}(w_1), >0} U_{P^+_{J'}}$. Since $Z_G(t)^0=L_{J'}$, any element in $t U^+_{w_1, >0}$ is conjugate to an element in $t U^+_{\pi_{J'}(w_1), >0}$. 

For any $u' \in U^+_{\pi_{J'}(w_1), >0}$, we have $t u'=u' t$. This is the Jordan decomposition of the element $t u'$. Since $\supp(\pi_{J'}(w_1))=J'$, we have that $u'$ is a regular unipotent element in $L_{J'}$. Thus $t u'$ is a regular element in $G$. So $g$ is also a regular element in $G$. 
\qed

\subsection{Regular semisimple conjugacy classes in $G_{>0}$}
Set $$T_{>1}=\{t \in T_{>0}; \a_i(t)>1 \text{ for all } i \in I\}.$$ Then any elemenet in $T_{>1}$ is regular and semisimple. By \cite[Theorem 5.6 \& Corollary 8.10]{Lu94}, any element $G_{>0}$ is conjugate to an element in $T_{>1}$. 

Now we prove that the converse is also true. This provides a stronger statement than Theorem \ref{main1} (1) for the cell $G_{>0}$. We also conjecture that the similar statement holds for any cell in $G_{\ge 0}$. 

\begin{theorem}
Let $t \in T_{>1}$ and $C$ be the regular semisimple conjugacy class of $t$ in $G$. Then $C \cap G_{>0} \neq \emptyset$. 
\end{theorem}

\begin{proof}
It suffices to consider the case where that $G$ is semisimple. Let $\mathfrak g$ be the Lie algebra of $G$. For $i\in I$, let $e_i\in\mathfrak g$ (resp. $f_i\in\mathfrak g$) be the differential of $x_i$ (resp. $y_i$). Let $\mathfrak h$ be the Cartan subalgebra of $\mathfrak g$ corresponding to $T$. Let $\mathfrak h_\BR$ be the set of real points of $\mathfrak h$ and let $\mathfrak h_+$ be the fundamental open Weyl chamber in $\mathfrak h_\BR$. Let $\exp:\mathfrak g \to G$ be the exponential map. By our assumption on $C$, we can find $y\in\mathfrak h_+$ such that $\exp(y)=t$.

Let $\mathfrak h/W$ be the set of orbits of the standard Weyl group action on $\mathfrak h$. For $x\in\mathfrak g$, the semisimple part of $x$ is $G$-conjugate to an element $x'\in\mathfrak h$ whose Weyl group orbit depends only on $x$; this defines a map $\CI:\mathfrak g \to \mathfrak h/W$. Following \cite{K}, we define the set of ``Jacobi elements'' 
$$\mathfrak g_{>0}=\{\sum_ia_if_i+h+\sum_ib_ie_i;a_i\in\BR_{>0},b_i\in\BR_{>0},h\in\mathfrak h_\BR\}.$$
By \cite[(3.5.26)]{K}, we have $\CI(\mathfrak g_{>0})=\CI(\mathfrak h_+)$. Hence any element in $\mathfrak g_{>0}$ is regular semisimple and contained in an $\BR$-split Cartan subalgebra and we can find $x\in\mathfrak g_{>0}$ such that the semisimple part of $x$ is $G$-conjugate to $y$; since $x$ is semisimple we see that $x$ is $G$-conjugate to $y$. Since $\exp(y)\in C$, we see that $\exp(x)\in C$. 

It remains to show that $\exp$ maps $\mathfrak g_{>0}$ into $G_{>0}$. This property is stated in \cite[5.8(c)]{Lu94} where a property close to it is proved (see \cite[5.9(c)]{Lu94}). But as mentioned in \cite[p.550, footnote]{Lu94}, a similar proof yields \cite[5.8(c)]{Lu94}. 
\end{proof}

\section{Jordan decomposition}

\subsection{Jordan decomposition}
Let $g \in G_{\ge 0}$ and $g=g_sg_u=g_ug_s$ be the Jordan decomposition of $g$ in $G$. According to \cite{Lu19} the centralizer $H=Z_G(g_s)$ of $g_s$ is the Levi subgroup of a parabolic subgroup of $G$; in particular it is connected.

Note that $g_s$ lies in the identity component of $H$. Since $g_s$ is a central element of $H$, we have that $g_s$ is automatically contained in $H_{\underline{\mathbf P_H}, \ge 0}$ associated to any weak pinning $\underline{\mathbf P_H}$ of $H$. 
We conjecture that there exists a weak pinning $\underline{\mathbf P_H}$ of $H$ so that the associated totally nonnegative part $H_{\underline{\mathbf P_H}, \ge 0}$ of $H$ contains the unipotent part $g_u$ of $g$ (and hence contains $g$ itself). Moreover, the desired weak pinning $\underline{\mathbf P_H}$ of $H$ has a representation-theoretic interpretation. 

To do this, we use the positivity property of the canonical basis for the simply-laced groups. For arbitrary group, we may apply the ``folding method'' of \cite{Lu94} once the desired result is proved for the simply-laced groups. 

By \cite[\S 9]{Lu19}, if $g \in G_{\ge 0}$, then $g_s$ is conjugate by an element in $G(\BR)$ to an element in $T_{>0}$. Let $\l$ be a dominant weight of $G$ and $V_\l$  be the corresponding irreducible finite dimensional representation of $G$. We have $V_\l=\oplus_{c\in\BR_{>0}}V_\l(g, c)$, where $V_\l(g, c)$ is the generalized eigenspace of $g:V_\l \to V_\l$ (see \cite[\S 9]{Lu19}). Now $H=Z_G(g_s)$ acts naturally on each $V_\l(g, c)$. Let $c_M=\max\{c\in\BR_{>0};V_\l(c)\neq 0\}$. We show that 

(a) {\it For any $g \in G_{\ge 0}$ and any dominant weight $\l$ of $G$, $V_{\l}(g, c_M)$ is an irreducible representation of $H$ with highest weight $\l$.}

We call an element $t \in T_{>0}$ a {\it standard element} if $\a_i(t) \ge 1$ for all $i \in I$. If $g_s$ is a standard element, then $H$ is a standard Levi subgroup $L_J$ for some $J \subset I$. By definition, $V_\l(g, c_M)$ is spanned by the weight vectors with $\mu$ such that $\mu(g_s)=\l(g_s)$. Note that the condition $\mu(g_s)=\l(g_s)$ holds if and only if $\l-\mu \in \sum_{i \in J} \BZ \a_i$. In this case, $V_\l(g, c_M)$ is the $L_J$-subrepresentation of $V_\l$ generated by the highest weight vectors of $V_\l$. In particular, $V_\l(g, c_M)$ is an irreducible finite dimensional representation of $H$ with the highest weight $\l$. 

Note that any element in $T_{>0}$ is conjugate to a standard element. Thus for any $g \in G_{>0}$, there exists $h \in G(\BR)$ such that $h g_s h \i \in T_{>0}$ and $\a_i(h g_s h \i) \ge 1$ for all $i \in I$. Set $g'=h g h \i$. Then $g'_s=h g_s h \i$, $Z_G(g'_s)=h H h \i$ and $V_{\l}(g', c_M)=h \cdot V_{\l}(g, c_M)$. Hence (a) is proved.

Now we state the conjectural totally nonnegative Jordan decomposition. 

\begin{conjecture}\label{conj}
Assume that $G$ is simply laced. Let $g \in G_{\ge 0}$ and $H=Z_G(g_s)$. Then there exists a weak pinning $\underline{\mathbf P_H}$ of $H$ such that 

\begin{enumerate}
    \item $g_u \in H_{\underline{\mathbf P_H}, \ge 0}$; 
    
    \item $V(g, c_M)_{\underline{\mathbf P_H}, \ge 0}=V(g, c_M) \cap V_{\ge 0}$ for any irreducible representation $V$ of $G$. 
\end{enumerate}
\end{conjecture}


By \S\ref{sec:weakb} and \S\ref{sec:weakp}, the weak pinning on $H$, if it exists, is uniquely determined by $V(g, c_M) \cap V_{\ge 0}$. Moreover, as the totally nonnegative part of $H$ is determined by a representation-theoretic method, it only depends on the semisimple part $g_s$, but not on the element $g$ itself. In particular, Conjecture \ref{conj} claims that the totally nonnegative part $H_{\ge 0}$ is ``canonical''.

\subsection{Non simply-laced groups} 
We follow \cite[\S 1.6]{Lu94}. Let $\tilde G$ be the simply connected (algebraic) covering of the derived group of $G$ and $\pi: \tilde G \to G$ be the covering homormorphism. The pinning on $G$ induces a pinning on $\tilde G$. Then there exists a semisimple, simply connected, simply laced group $\dot G$, a pinning on $\dot G$ and an automorphism $\s: \dot G \to \dot G$ compatible with the pinning such that there is an isomorphism $\dot G^\s \to \tilde G$ so that the associated pinning on $\dot G^\s$ (defined in \cite[\S 1.5]{Lu94}) is compatible with the pinning on $\tilde G$. In this case, we have $G_{\ge 0}=(Z_G)_{>0} \, \pi(\dot G^\s_{\ge 0})$. 

Let $g \in G_{\ge 0}$. Then there exists $z \in (Z_G)_{\ge 0}$ and $\dot g \in \dot G_{\ge 0}^\s$ such that $g=z \pi(\dot g)$. Then $g_s=z \pi((\dot g)_s)$ and $g_u=\pi((\dot g)_u)$. Set $H=Z_G(g_s)$ and $\dot H=Z_{\dot G}((\dot g)_s)$. Then $H=Z_G \, \pi(\dot H^\s)$. 

We assume that Conjecture \ref{conj} holds for $\dot G$. Since $\dot H=s(\dot H)$ and the totally nonnegative part $\dot H_{\ge 0}$ is ``canoncial'', we have in particular that $H_{\ge 0}$ is $\s$-stable. Set $H_{\ge 0}=(Z_G)_{>0} \, \pi(\dot H^\s_{\ge 0})$. By Conjecture \ref{conj}, $\dot g_u \in \dot H_{\ge 0}^\s$. Thus $$g_u=\pi(\dot g_u) \in \pi(\dot H^\s_{\ge 0}) \subset H_{\ge 0}.$$ 

As a summary, we have

(a) {\it Suppose that Conjecture \ref{conj} holds for $\dot G$. Then for any Levi subgroup $H$ of a parabolic subgroup of $G$, there is a ``canonical'' totally nonnegative part $H_{\ge 0}$ such that for any $g \in G_{\ge 0}$ with $Z_G(g_s)=H$, we have $g_u \in H_{\ge 0}$.}


\subsection{Some evidence of Conjecture \ref{conj}}
In this subsection, we verify Conjecture \ref{conj} in some special cases. We start with a lemma on $V_{\ge 0}$. 

\begin{lemma}\label{u-v}
Let $V$ be an irreducible finite dimensional representation of $G$ and $\b$ be the canonical basis of $V$. Let $J \subset I$ and $V_J \subset V$ be the $L_J$-submodule spanned by the highest weight vector of $V$. Let $V_{\ge 0}=\sum_{b \in \b} \BR_{\ge 0} b$ and $(V_J)_{\ge 0}=V_J \cap V_{\ge 0}$. Then for any $w \in W^J$ and $u \in U^-_{w, >0}$, we have $$u \cdot (V_J)_{\ge 0}=(u \cdot V_J) \cap V_{\ge 0}.$$
\end{lemma}

\begin{proof}
By \cite[\S 3.2]{Lu94}, $u \cdot V_{\ge 0} \subset V_{\ge 0}$. Thus $u \cdot (V_J)_{\ge 0} \subset (u \cdot V_J) \cap V_{\ge 0}$. Now we prove the other direction. By definition, $(V_J)_{\ge 0}=\sum_{b \in \b \cap V_J} \BR_{\ge 0} b$. We recall the following result. 

(a) {\it Let $i \in I$ and $b \in \b$ such that $x_i(a) \cdot b=b$ for all $a \in \BR$. Then $\dot s_i \cdot b \in \b$.}

This is a special case of the braid group action on the canonical basis \cite{Lu96}. See also \cite[Lemma 2.8]{BH}. 

Let $w=s_{i_1} \cdots s_{i_n}$ be a reduced expression. Since $w \in W^J$, for any $j$ we have $(s_{i_{j+1}} \cdots s_{i_n}) \i \a_{i_j} \notin \Phi_J$. By definition, $x_{i_j}(a) (s_{i_{j+1}} \cdots s_{i_n}) \cdot b=(s_{i_{j+1}} \cdots s_{i_n}) \cdot b$ for any $1 \le j \le n$, $a \in \BR$ and $b \in V_J$. Thus by (a), for any $b \in \b \cap V_J$, we have $\dot s_{i_n} \cdot b \in \b$, $\dot s_{i_{n-1}} \dot s_{i_n} \cdot b \in \b$, \ldots, $\dot w \cdot b \in \b$. 

By definition, \[\tag{b} u \cdot b \in \BR_{>0} \dot w \cdot b+\sum_{b' \in \b, \text{wt}(b')>\text{wt}(b), \text{wt}(b')-\text{wt}(b) \in \sum_{j=1}^n \BN s_{i_1} \cdots s_{i_{j-1}} \a_{i_j}} \BR_{\ge 0} b'.\] Since $w \in W^J$, $w(\Phi_J) \cap  \sum_{j=1}^n \BN s_{i_1} \cdots s_{i_{j-1}} \a_{i_j}=\{0\}$. In other words, the only element in $\{\dot w \cdot b'; b' \in \b \cap V_J\}$ that occurs in the right hand of (b) with positive coefficient is $\dot w \cdot b$.

Let $\sum_{b \in \b \cap V_J} c_b b \in V_J$ with $u \cdot \sum_{b \in \b \cap V_J} c_b b \in V_{\ge 0}$. Note that $$u \cdot \sum_{b \in \b \cap V_J} c_b b=\sum_{b \in \b \cap V_J} c_b u \cdot b \in \sum_{b \in \b \cap V_J} c_b \BR_{>0} \dot w \cdot b+\sum_{b' \in \b, b' \notin \dot w \cdot V_J} \BR_{\ge 0} b'.$$ Then $c_b \ge 0$ for all $b \in \b \cap V_J$. Thus $\sum_{b \in \b \cap V_J} c_b b \in (V_J)_{\ge 0}$. The Lemma is proved. 
\end{proof}

\begin{proposition}\label{prop:conj}
Let $w_1, w_2 \in W$ with $\supp(w_2) \subset \supp(w_1)$. Then for any $g \in G_{w_1, w_2, >0}$ or $g \in G_{w_2, w_1, >0}$, conjecture \ref{conj} holds. 
\end{proposition}

\begin{proof}
We prove the case where $g \in G_{w_1, w_2, >0}$. The case $g \in G_{w_2, w_1, >0}$ is proved in the same way. 

Let $J_i=\supp(w_i)$. By Lemma \ref{lem:w-to-1}, there exists $u_1 \in U^-_{w_{J_2}, >0}$ such that $u_1 \i g u_1 \in t U^+_{w_1, >0}$ for some $t \in T_{>0}$ with $\a_i(t)>1$ for all $i \in J_2$. By Lemma \ref{lem:conj}, there exists $u_2 \in U^-_{w_{t, J_1} \i, >0}$ such that $u_2 \i u_1 \i g u_1 u_2 \in \bar t_{J_1} U^+_{w_1, >0}$. By definition, $w_{t, J_1} \in W_{J_1} \cap W^{J_2}$. Thus $\ell(w_{J_2} w_{t, J_1} \i)=\ell(w_{J_2})+\ell(w_{t, J_1})$. We set $u_-=u_1 u_2 \in U^-_{w_{J_2} w_{t, J_1} \i, >0}$. It is also easy to see that there exists $u_+ \in L_{J_1} \cap U_{P^+_{I(\bar t_{J_1})}}$ such that 
\begin{itemize}
    \item the semisimple part $u_+ \i u_- \i g_s u_- u_+$ of $u_+ \i u_- \i g u_- u_+$ is $\bar t_{J_1}$;
    \item the unipotent part $u_+ \i u_- \i g_u u_- u_+$ of $u_+ \i u_- \i g u_- u_+$ is in $U^+_{\ge 0} \cap L_{I(\bar t_{J_1})}$. 
\end{itemize} 

Let $\bar t=\bar t_I$ and $J=I(\bar t)$. Let $w$ be the unique element in ${}^{J_1} W^J$ with $\bar t=\dot w \i \bar t_{J_1} \dot w$. Hence the semisimple part of $\dot w \i u_+ \i u_- \i g u_- u_+ \dot w$ is $\bar t$. Let $H=Z_G(g_s)$. Since $Z_G(\bar t)=L_J$, we have $H=(u_- u_+ \dot w) L_J (u_- u_+ \dot w) \i$. The pinning $\mathbf P_H$ of $H$ is defined to be the conjugate of the standard pinning of $L_J$ by the element $u_- u_+ \dot w$. In particular, $H_{\underline{\mathbf P_H}, \ge 0}=(u_- u_+ \dot w) L_{J, \ge 0} (u_- u_+ \dot w) \i$. 

Since $\bar t=\dot w \i \bar t_{J_1} \dot w$, we have $\Phi_{I(\bar t_{J_1})} =w(\Phi_{J}) \cap \Phi_{J_1}$. Since $w \in {}^{J_1} W^J$, we have $w \i(I(\bar t_{J_1})) \subset J$. Therefore $\dot w \i (U^+_{\ge 0} \cap L_{I(\bar t_{J_1})}) \dot w \subset U^+_{\ge 0} \cap L_J$. Thus $$g_u \in (u_- u_+) (U^+_{\ge 0} \cap L_{I(\bar t_{J_1})}) (u_- u_+) \i \subset (u_- u_+ \dot w) L_{J, \ge 0} (u_- u_+ \dot w) \i=H_{\underline{\mathbf P_H}, \ge 0}.$$ 

Part (1) of Conjecture \ref{conj} is proved. 

Let $V$ be an irreducible finite dimensional representation of $G$ and $v$ be a highest weight vector of $V$. Let $V_J \subset V$ be the $L_J$-submodule generated by $v$. Then $V(\bar t, c_M)=V_J$ and $V(g, c_M)=u_- u_+ \dot w \cdot V_J$. The nonnegative part $V(g, c_M)_{\underline{\mathbf P_H}, \ge 0}$ of $V(g, c_M)$ determined by the pinning $\mathbf P_H$ of $H$ equals $u_- u_+ \dot w \cdot V_{J, \ge 0}$. Since $\Phi_{I(\bar t_{J_1})}=w(\Phi_{J}) \cap \Phi_{J_1}$ and $w \in {}^{J_1} W$, we have $\dot w \i u_+ \dot w \in U_{P^+_J}$. Thus $u_- u_+ \dot w \cdot V_J=u_- \dot w u \cdot V_J=u_- \dot w \cdot V_J$ for some $u \in U_{P^+_J}$. Thus $$V(g, c_M)_{\underline{\mathbf P_H}, \ge 0}=u_- \dot w \cdot V_{J, \ge 0}.$$

By the proof of Lemma \ref{u-v}, $\dot w \cdot V_{J, \ge 0} \subset V_{\ge 0}$. Since $u_- \in U^-_{\ge 0}$, we have $u_- \cdot V_{\ge 0} \subset V_{\ge 0}$. Thus $V(g, c_M)_{\underline{\mathbf P_H}, \ge 0} \subset V_{\ge 0}$. 

On the other hand, let $\b$ be the canonical basis of $V$. Suppose that $\sum_{b \in \b \cap V_J} c_b b \in V_J$ with $u_- \dot w \cdot \sum_{b \in \b \cap V_J} c_b b \in V_{\ge 0}$. By (b) in the proof of Lemma \ref{u-v}, we have $$u_- \dot w \cdot \sum_{b \in \b \cap V_J} c_b b \in \sum_{b \in \b \cap V_J} c_b \BR_{>0} \dot w' \cdot b+\sum_{b' \in \b, \b' \notin \dot w' \cdot V_J} \BR_{\ge 0} b'.$$ Thus $c_b \ge 0$ for all $b \in \b \cap V_J$ and $\sum_{b \in \b \cap V_J} c_b b \in (V_J)_{\ge 0}$. 

Therefore $V(g, c_M)_{\underline{\mathbf P_H}, \ge 0}=u_- \dot w \cdot V_{J, \ge 0}=u_- \dot w \cdot V_J \cap V_{\ge 0}=V(g, c_M) \cap V_{\ge 0}$. Part (2) of Conjecture \ref{conj} is proved. 
\end{proof}

\subsection{Verification of conjecture \ref{conj} for $GL_3$} In this subsection, we prove conjecture \ref{conj} for $G=GL_3$. Let $W=S_3$ be the Weyl group of $G$ and $\{s_1, s_2\}$ be the set of simple reflections of $W$. By Proposition \ref{prop:conj}, if $g \in G_{\ge 0}$ but $g \notin G_{s_1, s_2, >0} \cup G_{s_2, s_1, >0}$, then the conjecture \ref{conj} holds for $g$. Now we prove the conjecture \ref{conj} for $g \in G_{s_1, s_2, >0}$. The elements in $G_{s_2, s_1, >0}$ are handled in the same way. 

The element $g$ can be written as $g=t u_1 u_2$ for some $u_1 \in U^+_{s_1, >0}$, $u_2 \in U^-_{s_2, >0}$ and $t \in T_{>0}$. 

If $\a_1(t) \neq 1$ and $\a_2(t) \neq 1$, then $g$ is regular semisimple and the conjecture \ref{conj} is obvious for $g$. 

If $\a_1(t)=\a_2(t)=1$, then $g_s=t$ and $H=G$. In this case, the weak pinning on $H$ is the equivalance class of the pinning on $G$ we fixed in the beginning. The conjecture \ref{conj} holds for $g$. 

It remains to consider the case where $\a_1(t)=1$ and $\a_2(t) \neq 1$. (The case where $\a_1(t) \neq 1$ and $\a_2(t)=1$ is proved in the same way.) In this case, there exists $a \in \BR$ such that $y_2(-a) g y_2(a)=t u_1$. We have $Z_G(t)=L_{\{1\}}$ and $H=y_2(a) L_{\{1\}} y_2(-a)$. The pinning on $G$ induces the pinning $\mathbf P_{\{1\}}=(T, B^+ \cap L_{\{1\}}, B^- \cap L_{\{1\}}, x_1, y_1)$ on the Levi subgroup $L_{\{1\}}$. The pinning $\mathbf P_H$ on $H$ is obtained from $\mathbf P_{\{1\}}$ by conjugating $y_2(a)$. Moreover $V(g_s, c_M)=y_2(a) \cdot V(t, c_M)$ and $V(g_s, c_M)_{\underline{\mathbf P_H}, \ge 0}=y_2(a) \cdot V(t, c_M)_{\ge 0}$.

If $\a_2(t)<1$, then $V(t, c_M)$ is the $L_{\{1\}}$-submodule of $V$ generated by the lowest weight vector of $V$ and $V(g_s, C_M)=V(t, c_M)$. In this case, $V(g_s, c_M)_{\underline{\mathbf P_H}, \ge 0}=V(t, c_M)_{\ge 0}=V(g_s, C_M) \cap V_{\ge 0}$. 

If $\a_2(t)>1$, then $V(t, c_M)$ is the $L_{\{1\}}$-submodule of $V$ generated by the highest weight vector of $V$. Moreover, we have $a>0$ since $u_2 \in U^-_{s_2, >0}$. By Lemma \ref{u-v}, $V(g_s, c_M)_{\underline{\mathbf P_H}, \ge 0}=y_2(a) \cdot V(t, c_M)_{\ge 0}=V(g_s, C_M) \cap V_{\ge 0}$. 

This finishes the verification of the conjecture \ref{conj} for $G=GL_3$.

\end{document}